\documentclass{amsart}
\usepackage{amsmath, amsthm, amssymb, amscd, mathrsfs, eucal, epsfig}
\usepackage{pinlabel}

\DeclareMathOperator{\PSL}{\mathrm{PSL}}
\DeclareMathOperator{\GL}{\mathrm{GL}}

\input xy
\xyoption{all}

\usepackage{hyperref}

\begin{document}

\newtheorem{theorem}{Theorem}[subsection]
\newtheorem{lemma}[theorem]{Lemma}
\newtheorem{corollary}[theorem]{Corollary}
\newtheorem{conjecture}[theorem]{Conjecture}
\newtheorem{proposition}[theorem]{Proposition}
\newtheorem{question}[theorem]{Question}
\newtheorem{problem}[theorem]{Problem}
\newtheorem*{PL_LO_thm}{PL is LO~\ref{thm:PL_LO}}
\newtheorem*{claim}{Claim}
\newtheorem*{criterion}{Criterion}
\theoremstyle{definition}
\newtheorem{definition}[theorem]{Definition}
\newtheorem{construction}[theorem]{Construction}
\newtheorem{notation}[theorem]{Notation}
\newtheorem{convention}[theorem]{Convention}
\newtheorem*{warning}{Warning}

\theoremstyle{remark}
\newtheorem{remark}[theorem]{Remark}
\newtheorem{example}[theorem]{Example}
\newtheorem*{case}{Case}

\def\C{\mathbb C}
\def\F{\mathfrak F}
\def\H{\mathbb H}
\def\R{\mathbb R}
\def\Z{\mathbb Z}
\def\RP{\mathbb{RP}}
\def\PL{\textnormal{PL}}
\def\GL{\textnormal{GL}}
\def\Homeo{\textnormal{Homeo}}
\def\a{{\alpha}}
\def\b{{\beta}}
\def\s{{\sigma}}
\def\G{{\Gamma}}
\def\Homeo{\textnormal{Homeo}}
\def\Diff{\textnormal{Diff}}
\def\fix{\textnormal{fix}}
\def\fro{\textnormal{fro}}
\def\Id{\textnormal{Id}}

\title{Low-dimensional topology and ordering groups}

\author{Dale Rolfsen}
\address{Department of Mathematics \\ University of British Columbia \\ Vancouver \\ Canada}
\email{rolfsen@math.ubc.ca}

\date{\today}

\begin{abstract}
This expository paper explores the interaction of group ordering with topological questions, especially in dimensions 2 and 3.  Among the topics considered are surfaces, braid groups, 3-manifolds and their structures such as foliations and mappings between them.  A final section explores currently ongoing research regarding spaces of homeomorphisms and their orderability properties.   This is not meant to be a comprehensive survey, but rather just a taste of the rich relationship between topology and the theory of ordered groups.
\end{abstract}

\maketitle

\section{Introduction}

Group theory and low-dimensional topology are closely related subjects.  Indeed the extremely active field of geometric group theory is essentially a marriage of the two.  Groups arise, for example in consideration of symmetries, self-mappings of a space, isometries and similar actions,
mapping class groups, braid groups, as well as in algebraic invariants of spaces such as the fundamental group.  Closed surfaces, for example, are determined by their fundamental groups.
As is well-known, every finitely presented group $G$ can be realized as the fundamental group of a finite 2-dimensional complex, constructed by taking a bouquet of oriented loops, one for each generator, and attaching disks to the loops according to the relators.  Indeed a group is free if and only if it is the fundamental group of a graph, i.e a 1-dimensional complex.  This is one way to see directly the important theorem that subgroups of free groups are free.

There is an intimate connection between the topology and geometry of three-dimensional manifolds and their fundamental groups, although the situation is somewhat more complicated than in two dimensions.  There are examples of closed 3-manifolds which are topologically distinct, but which have isomorphic fundamental groups -- e. g. certain lens spaces. 
Although it is well-known that all finitely presented groups arise as fundamental groups of closed 
4-dimensional manifolds, this is not the case in dimension three.  The algebraic characterization of groups which do arise as $\pi_1(M)$ for some 3-manifold $M$ is an open problem.   Yet the fundamental group does serve as a powerful 3-manifold invariant.  The most famous example is the Poincar\'e conjecture, now a theorem of Perelman, that $S^3$ is the only closed 3-manifold which has trivial fundamental group.
 
Certain structures on 3-manifolds are reflected in the algebraic structure of the fundamental group.  For example, Thurston \cite{Thur} described eight possible geometric structures for 3-manifolds and gave a procedure for determining the structure from the manifold's fundamental group. 

\section{Orderable groups}

If the elements of a group $G$ can be given a strict total order $<$ which is invariant under 
left-multiplication, so that $g < h$ implies $fg < fh$ for all $f, g, h \in G$, we call $<$ a {\em left-order} and say that $G$ is {\em left-orderable}.   If $<$ is also right-invariant, we will call it a {\em bi-order} and say that $G$ is {\em bi-orderable}.  In the classical literature, the usual term is ``orderable" --  we use the term bi-orderable for emphasis.

A left-order is determined by its positive cone $P = P_< := \{g \in G | 1 < g\}$ where $1$ denotes the identity of $G$.  Note that $P$ is closed under multiplication and that for every $g \in G$ exactly one of $g \in P, g^{-1} \in P$ or $g = 1$ holds.  On the other hand, any subset $P \subset G$ satisfying these conditions determines a left-order by declaring $g < h$ if and only if $g^{-1}h \in P$.  Note that a group is left-orderable if and only if it is right-orderable.  The criterion $gh^{-1} \in P$ defines a right-order with the same positive cone.

A group $G$ is {\em indicable} if there exists a surjective homomorphism of $G$ to the integers 
$\Z$.  If every nontrivial finitely generated subgroup of a group is indicable, the group is said to be {\em locally indicable}.  The following is well known.  We refer the reader to 
\cite{Fuc}, \cite{Glass}, \cite{MR77}, \cite{KM96} for further information on ordered groups.

\begin{proposition}
Bi-orderable $\implies$ locally indicable $\implies$ left-orderable $\implies$ torsion-free.  
\end{proposition}

The first implication follows from Theorem \ref{holder} a classical theorem of H\"older \cite{Holder01}.  Take a finitely generated bi-ordered group $H$  generated by 
$\{g_1, \cdots, g_k\}$ and consider its convex subgroups, which are linearly ordered by inclusion.  A subset $X$ of an ordered set is called {\em convex} if the conditions $x < y < z$ and $x, z \in X$ imply that $y \in X$.   We may assume without loss that 
$1 < g_1 <  \cdots < g_k$ and that no fewer elements of the group will generate $H$.  We then consider the set $C$ which is the union of all convex subgroups of $H$ which do {\em not} contain $g_k$ and argue that $C$ is a normal convex subgroup and that $H/C$ inherits a bi-ordering which is Archimedian.  An {\em Archimedian} ordered group is one in which powers of each nonindentity element are cofinal in the ordering.  By H\"older's theorem, one can find an injective homomorphism $H/C \to \R$ and since $H/C$ is finitely generated, there is a surjective homomorphism $H/C \to \Z$.  Then the composite surjection $H \to H/C \to \Z$ shows that $H$ is indicable, and so bi-orderable groups are locally indicable.

That local indicability implies left-orderability is a theorem of Burns and Hale \cite{BH72}, which we state later (Theorem \ref{BH}) but will not prove here.  Left orderability implies no torsion, for if $g > 1$, then $g^2 > g > 1$ and inductively $g^n > 1$ for all positive powers $n$.  A similar argument applies if $g < 1$, so no nonidentity element can have finite order.\qed

None of these implications is reversible -- examples will be discussed later in this article.  Each of the properties named in the above proposition is a local property, that is, a group has that property if and only if every finitely generated subgroup has the property.  Thus, for example, all torsion-free abelian groups are bi-orderable, since it is easy to bi-order $\Z^n$ as an additive group.  Indeed, the lexicographic order is a bi-order, but there are uncountably many others.  One may take any 
vector ${\bf v} \in \R^n$ with coordinates independent over the rationals and order vectors in $\Z^n$ according to their dot product with  ${\bf v}$.  Sikora \cite{Sikora04} showed that $\Z^n$
has no `isolated' orderings in the sense that for any {\em finite} set of inequalities that hold in a given ordering, there exist {\em different} orderings in which that set of relations still hold.

It should be remarked that a group is locally indicable if and only it supports a Conradian left-order, as proved in \cite{Bro84} and with different proofs in \cite{RhemR} and \cite{Nav10b}.
A left-order is said to be {\em Conradian} if for every pair of positive elements $g, h$ there exists a positive integer $n$ such that $g < hg^n$.  This concept was introduced by Conrad in \cite{Conrad59} in the (equivalent) context of right-orders.
Recently it was shown by Navas \cite{Nav} Proposition 3.7, that $n=2$ is sufficient; that is an order is Conradian if and only if for every positive $g, h$ the inequality $g<hg^2$ holds.

\begin{theorem}[H\"older]\label{holder}
If $G$ is an Archimedian bi-ordered group, then there is an embedding $G \to \R$ which is simultaneously order-preserving and an isomorphism into the additive real numbers.  In particular, $G$ must be Abelian.
\end{theorem}

A convenient way of ordering a group $G$ is to find a normal subgroup $K$ which is orderable and such that the quotient $H = G/K$ is also orderable.  That is, we consider an exact sequence
$$1 \to K \to G \to H \to 1$$

\begin{proposition}\label{exact}
If $K$ and $H$ are left-orderable, then so is $G$.  If $K$ and $H$ are bi-orderable, then $G$ is bi-orderable if and only if there exists a bi-order of $K$ invariant under conjugation by $G$.
\end{proposition}

The positive cone for $G$ can be taken to be the union of the positive cone for $K$ and the pullback of the positive cone for $H$.

Left-orderability is closely related to dynamical properties of a group.

\begin{proposition} \label{action}
If a group acts effectively on the real line $\R$ by orientation-preserving homeomorphisms, then it is left-orderable.
\end{proposition}

{\em Effective} means that only the identity element of the group acts via the identity map.  Of course orientation-preserving in this context is equivalant to order-preserving.  The hypothesis of the proposition says that the group embeds isomorphically in $\Homeo_+(\R)$, the set of all 
orientation-preserving homeomorphisms of $\R$, with composition as the group operation.  Since orderability is obviously inherited by subgroups, the proposition follows from the following lemma.
\qed

\begin{lemma}
$\Homeo_+(\R)$ is left-orderable.
\end{lemma}

This can be seen by enumerating a dense subset  $x_1, x_2, \dots$ in $\R$, then comparing two functions $f$ and $g$ in $\Homeo_+(\R)$ by declaring that $f \prec g$ if and only if 
$f(x_i) < g(x_i)$ at the first $i$ at which $f(x_i) \ne g(x_i)$.  One checks routinely that $\prec$ left-orders $\Homeo_+(\R)$.
The group $\Homeo_+(\R)$ is universal for countable left-orderable groups, in the following sense.  

\begin{proposition} \label{universal}
Every countable left-ordered group $G$ is isomorphic with a subgroup of $\Homeo_+(\R)$.
\end{proposition}

There is a well-known construction to build such an isomorphism, called a `dynamic realization' of $G$.  We refer the reader to one of the textbooks already mentioned or the excellent article by Navas \cite{Nav} on the subject.

Bi-orderable groups have special algebraic properties, which we will use in the sequel.

\begin{lemma}\label{unique roots}
In a bi-ordered group one has unique roots.  That is, if $g^n = h^n$ for some positive integer $n$, then $g=h$.
\end{lemma}

This is true because in a bi-ordered group (though not necessarily in a left-ordered group) it is easy to show that one can multiply inequalities: $g < h$ and $g' < h'$ imply that $gg' < hh'$.  In particular $g<h$ implies that $g^2<h^2$ and then $g^3<h^3$, etc., so that distinct elements can never have equal $n$-th powers. \qed

\begin{lemma}\label{power commutes}
In a bi-ordered group if an element $g$ commutes with a nonzero power $h^n$ of some element $h$, then $g$ commutes with $h$ itself.
\end{lemma}

If $g$ and $h$ do not commute suppose without loss that $n > 1$ and $h < g^{-1}hg$.  Then multiplying this inequality by itself repeatedly we conclude the contradiction $h^n < g^{-1}h^ng$. \qed 

\begin{proposition} \label{freeproduct}
If $G$ and $H$ are left-orderable (resp. bi-orderable), then the same is true of their free product
$G * H$.  In particular, free groups are bi-orderable.
\end{proposition}

This is a nontrivial result of Vinogradov \cite{Vin} which we will not prove here.  However, it is instructive to see how one might bi-order a free group $F$ of finite rank.  The lower central series
$\G_1(F) \supset \G_2(F) \supset \cdots$ is defined by $\G_1(F) = F$, and for $i \ge 1$ let 
$\G_{i+1}(F) := [F, \G_i(F)]$, the group generated by commutators $f^{-1}g^{-1}fg$ with $f \in F$ and 
$g \in \G_i(F)$.  It is well-known that for free nonabelian groups the lower central series does not terminate, but $\cap_{i=1}^\infty \G_i(F) = \{1\}$ and the lower central quotients
$\G_i(F) / \G_{i+1}(F)$ are free abelian groups of finite rank, which we know how to bi-order.  Indeed, if we choose arbitrary orders on these quotients, then we obtain a positive cone for a  bi-order of $F$ by declaring a nonidentity element $f$ to be positive if and only if its class in $\G_i(F) / \G_{i+1}(F)$ is positive in the chosen ordering, where $\G_i(F)$ is the last group in the series that contains $f$. \qed

Regarding the number of possible orders on a group, there is a curious difference between left- and bi-orderability.

\begin{proposition}
A left-orderable group has either finitely many left-orders or the number of left-orders is uncountable.  However, there exist bi-ordered groups which support exactly a countable infinity of bi-orders. 
\end{proposition}

The first sentence is a deep theorem of Linnell \cite{Lin}, and Buttsworth \cite{Buttsworth71} gives  examples verifying the second. \qed

\section{Braid groups}

A beautiful connection between topology and algebra is through Artin's braid groups.  For each positive integer $n$ one considers $n$ strings in 3-space which are monotone in one direction, and disjoint, but possibly intertwined, and begin and end at specified points in two parallel planes.   The product of braids is concatenation, as illustrated in Figure 1.  Two braids are equivalent if one deforms to the other through a one-parameter family of braids, with endpoints fixed at all times.   The identity in the $n$-strand braid group $B_n$ is represented by a braid with no crossings -- the strands can be taken as straight lines.

According to Artin \cite{Artin} for each $n \ge 2$,  $B_n$  has generators $\s_1, \dots, \s_{n-1}$, in which $\s_i$ is the simple braid in which all the strands are straight, except that the strand labelled $i$ crosses over the strand labelled $i+1$.  These generators are subject to the relations $\s_i\s_j = \s_j\s_i$ if $|i-j| > 1$ and 
$\s_i\s_j\s_i = \s_j\s_i\s_j$ when $|i-j| = 1$.

Each $n$-strand braid has an associated permutation of the set $\{1, \dots, n\}$ which records how the strands connect the endpoints of the various strands.  In other words, there is a homomorphism
$B_n \to S_n$, where $S_n$ denotes the symmetric group on $n$ elements, in which 
$\s_i$ is sent to the simple permutation interchanging $i$ and  $i+1$.  This homomorphism is surjective -- it is easy to see that any permutation of $\{1, \dots, n\}$ can be realized by infinitely many braids (if $n > 1$).

\begin{figure}[htpb]
\centering
\includegraphics[scale=0.85]{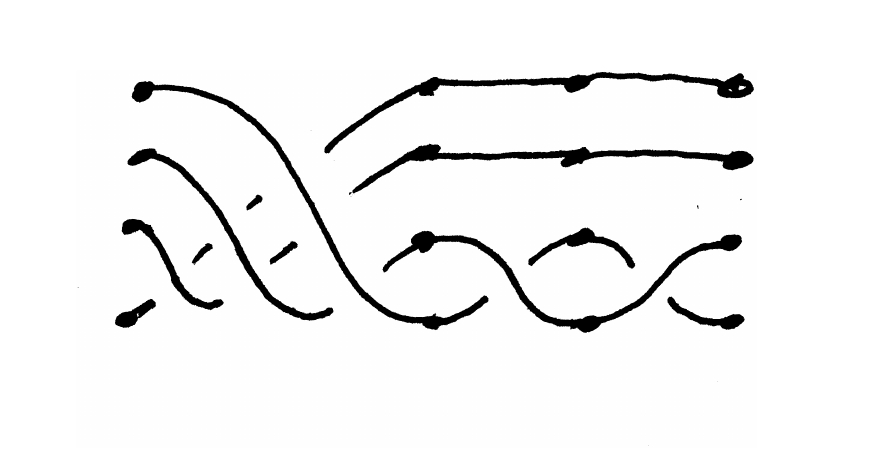}
\label{braid}
\caption{The product of $\Delta_4, \s_1$ and $\s_1^{-1}$ in $B_4$}
\end{figure}

My own interest in orderable groups began when I learned of Patrick Dehornoy's \cite{Deh}  remarkable proof of Theorem \ref{braidsLO}.

  I had been working on a conjecture of J. Birman, which involved computations in the group ring 
$\Z B_n$.  In those calculations one always had to worry about zero divisors (for example to argue that $ab = ac \implies b=c$). Although it was well-known that $B_n$ is torsion-free, I didn't know if there were any zero divisors in $\Z B_n$.   My worries were over when I learned of Dehornoy's ordering of $B_n$ and the fact that the group rings of left-orderable groups do not have zero divisors.  Whether the group ring of an arbitrary torsion-free group has zero divisors is still an open question.  The book \cite{DDRW} discusses the ordering of braid groups in great detail; I will only sketch the ideas of the ordering here.

\begin{theorem}\label{braidsLO}
The braid groups $B_n$ are left-orderable. 
\end{theorem}

Dehornoy defines a positive cone $P$ for $B_n$ in the following manner.  If $\b \in B_n$, call it {\em $i$-positive} if $\b$ has an expression as a word in the generators in which $\s_i$ appears with only positive exponents, and no $\s_j$ appears for $j<i$.  The positive cone is defined to be the set of all braids which are $i$-positive for some $i \in \{1, \dots, n-1\}$.  It is clear that $P$ is closed under multiplication.  The tricky part is to show that for any $\b \ne 1$ exactly one of $\b$ and $\b^{-1}$ is $i$-positive for some $i$.  The interested reader should consult \cite{Deh} or \cite{DDRW} for details.

Other methods of constructing left-orders of $B_n$, using its interpretation as a mapping class group, were later given in \cite{FGRRW99} and  \cite{SW00}.  The latter paper constructs infinitely many left-orders of $B_n$ (for $n>2$) using the Nielsen-Thurston theory of automorphisms of surfaces to show that $B_n$ acts on the real line.

\begin{proposition}\label{braidsnotbo}
For $n \ge 3$ the braid group $B_n$ is not bi-orderable.
\end{proposition}

To see this, consider the generators $\s_1$ and $\s_2$.  We note that $\s_1\s_2$ is not equal to $\s_2\s_1$.  In fact their associated permutations $(132)$ and $(123)$, respectively, are different.
On the other hand we have, using the braid relation,
$$(\s_1\s_2)^3 = (\s_1\s_2\s_1)(\s_2\s_1\s_2) = (\s_2\s_1\s_2)(\s_1\s_2\s_1) = (\s_2\s_1)^3.$$
If the braid group were bi-ordered, this equation would contradict the Lemma \ref{unique roots}. 

We note that $B_3$ is an example of a locally indicable group which is not bi-orderable.

An important class of braids are the {\em pure} braids, that is the kernel of the homomorphism   
$B_n \to S_n$ mentioned above.  They form a normal subgroup $P_n$ of $B_n$ of index $n!$.

\begin{theorem}\label{pure}
The pure braid groups $P_n$ are bi-orderable. 
\end{theorem}

This can be proved inductively as follows.  Since $P_1 = \{1\}$ and $P_2 \cong \Z$, the base of the induction is clear.  Consider $P_n$ for $n>2$.  By deleting the last strand, one defines a surjective homomorphism $h: P_n \to P_{n-1}$.  As shown by Artin, the kernel of $h$ is a free group $F$ of rank $n-1$.   Free groups are bi-orderable and we may assume for induction that $P_{n-1}$ is also bi-orderable.  According to Proposition \ref{exact} it remains to show that $F$ can be ordered in such a way that the conjugation action of $P_n$ upon $F$ preserves the order.  This can be done, basically because the action reduces to the trivial action when one passes to the abelianization of $F$.  Details may be found in \cite{RZ} or \cite{KR}.

Alternatively, one may prove Theorem \ref{pure} by appealing to the result of Falk and Randall \cite{FR} that pure braid groups have a lower central series with the same properties cited above for free groups, and then bi-order them in the same way. \qed

The methods of ordering $P_n$ and left-ordering $B_n$ which we have described are very different.  One might ask if one can find compatible orderings.  The answer is ``no'' as shown independently in \cite{RhemR} and \cite{DD}.

\begin{theorem}
No left-order of $B_n$ restricts to a bi-order of $P_n$.
\end{theorem}

Another interesting fact about ordering of the braid groups is that for $n \ge 3$, there are isolated orderings of $B_n$, as shown by Dubrovina and Dubrovin \cite{DD}.  For example in $B_3$, the semigroup generated by $\s_1\s_2$ and $\s_2^{-1}$ is actually a positive cone for a left-ordering.  It follows that it is the {\em unique} ordering of $B_3$ in which those two braids are greater than the identity. 

Interestingly, it is shown in \cite{NW11} that none of the Nielsen-Thurston orders are isolated.  Indeed many of them (including the Dehornoy order) are limit points of their conjugates in the space of left-orders.  A conjugate $<_g$ of an order $<$ of a group is defined by $x <_g y \iff
gxg^{-1} < gyg^{-1}$, where $g$ is a fixed group element.

\section{Knots and braids}

We will present a very brief review of the theory of knots and an interesting application of the braid ordering to knot theory.

There is a close connection between the braid groups and the theory of knots and links.   By a knot we mean an embedding of a circle in 3-dimensional space $\R^3$.  This models the idea of a knotted rope, but we assume the ends of the rope are attached to each other, preventing the knot from being untied by simply pulling the rope through itself.  Two knots are considered equivalent if there is an isotopy (continuous family of homeomorphisms) of $\R^3$ which at the end takes one knot to the other.  One can add knots $K_1$ and $K_2$ by tying them in distant parts of the rope, thus forming the connected sum $K_1 \sharp K_2$.  Figure \ref{knotsum} illustrates this.

\begin{figure}[htpb]
\centering
\includegraphics[scale=0.75]{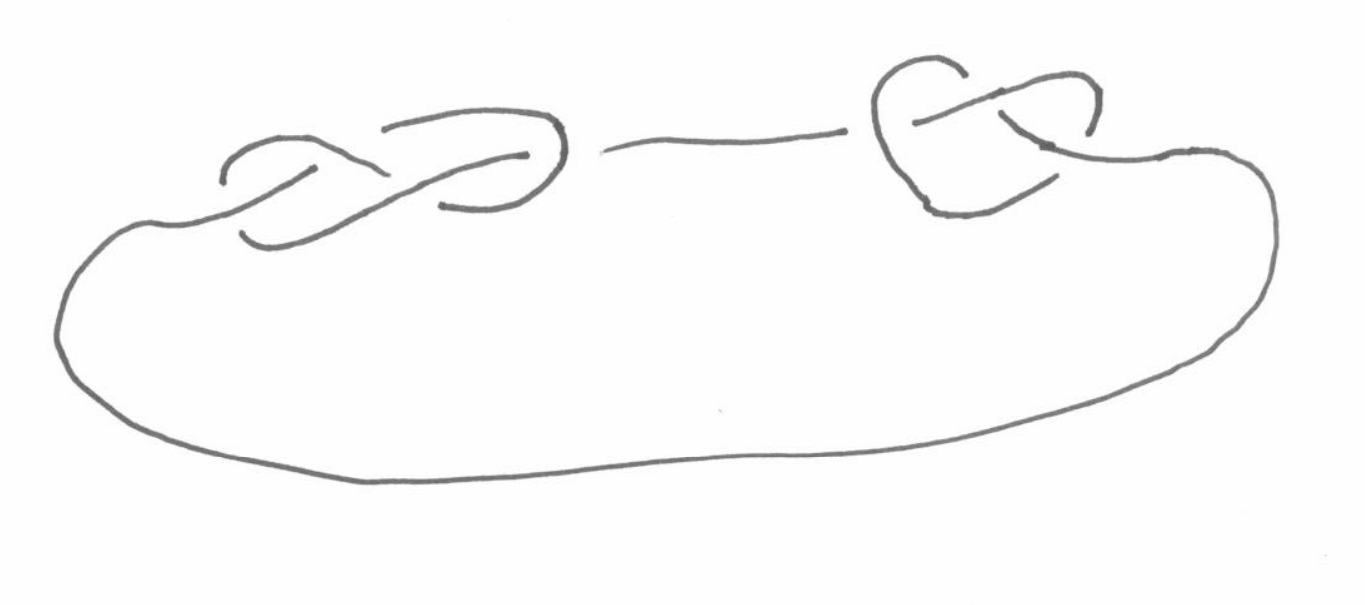}
\caption{The sum $4_1 \; \sharp \; 3_1$ of the figure-eight and the trefoil.}
\label{knotsum}
\end{figure}

The sum is easily seen to be commutative and associative, up to equivalence.  Thus the set of (equivalence classes of) knots forms an abelian semigroup, with unit the trivial knot, or {\em unknot}, which is a curve equivalent to a round circle in space.  There is a prime decomposition theorem: any knot $K$ can be written $K \cong K_1 \sharp \cdots \sharp K_p$ where each $K_i$ is {\em prime}, i.e. not expressible as a sum of two nontrivial knots.  Moreover, in this decomposition the terms are unique up to order.   Finally, it is a theorem that there are no inverses: if 
$K_1 \sharp K_2$ is equivalent to the unknot, then so are both $K_1$ and $K_2$.  This is one reason that braids are convenient in the study of knots, as the braids do form groups: every braid has an inverse -- namely its mirror image in a plane perpendicular to the direction in which the strands are monotone.

If $\b$ is a braid, its {\em closure} $\hat{\b}$ is a knot (or disjoint union of knots, called a link) formed by connecting the ends as indicated in Figure 3.  

\begin{figure}[htpb]
\centering
\includegraphics[scale=0.75]{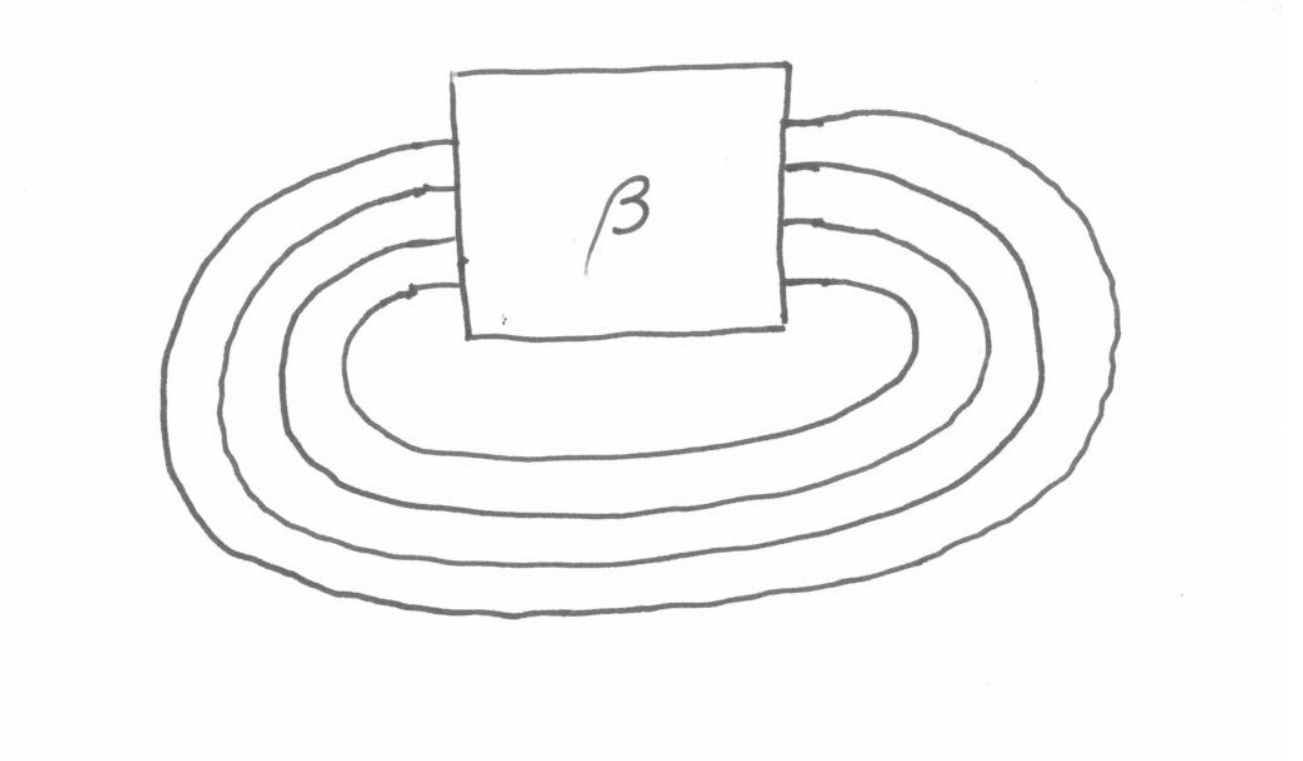}
\label{closure}
\caption{The closure of a braid}
\end{figure}

Many interesting properties of knots have arisen from this correspondence.  For example the Jones polynomial \cite{Jones} of a knot was discovered by considering a certain family of representations of the braid groups.  There is an interesting application of the Dehornoy braid order to knot theory due to Malyutin and Netsvetaev \cite{MN}.

The $n$-strand braid $\Delta_n$ is defined by the equation
$$\Delta_n = (\s_1\s_2\cdots \s_{n-1})(\s_1\s_2\cdots \s_{n-2}) \cdots (\s_1\s_2)(\s_1)$$
and corresponds to the braid formed by taking $n$ parallel strands and giving them a half-twist, as in Figure 1 for the 4-strand case.  The center of $B_n$, for $n\ge 3$ turns out to be exactly the cyclic subgroup generated by $\Delta_n^2$.

\begin{theorem}[Malyutin and Netsvetaev]
Suppose $\b \in B_n$ is a braid whose closure $\hat{\b}$ is a knot.  Assume that in the Dehornoy ordering of $B_n$ one has either $\b > \Delta_n^4$ or $\b < \Delta_n^{-4}$.  Then $\hat{\b}$ is a nontrivial prime knot.
\end{theorem}

Other applications to knot theory have been found by Ito \cite{Ito} which gives a lower bound on the {\em genus} of a knot (a measure of its complexity, c.f. \cite{KL}) which is the closure of a braid, in terms of the braid's place in the Dehornoy ordering.  The connection between braid groups and knot theory has had profound applications, and I believe the orderability of braids will have further implications in knot theory and related areas of topology.

\section{Surface groups}

We recall that a topological space is called an $n$-manifold if every point has a neighborhood homeomorphic with $\R^n$, Euclidean $n$-space.  More generally an $n$-manifold with boundary may be locally modeled on the closed half-space $R^n_+$.  A compact manifold without boundary is called a {\em closed} manifold.  A {\em surface} is a 2-manifold, possibly with boundary.

If a connected surface is noncompact, or has nonempty boundary, then its fundamental group is a free group.  As already mentioned free groups are known to be bi-orderable.   It remains to consider closed surfaces.  There is an operation called connected sum of two surfaces: one removes a small open disk from each surface and attaches the two remaining surfaces by identifying the $S^1$ boundaries where the disks were removed.  It is classical that closed surfaces have been classified as follows.  

Orientable surfaces are the sphere $S^2$, the torus $T^2$, and connected sums of $g$ tori, called the orientable surface of genus $g$.  

The simplest nonorientable surface is the projective plane $\RP^2$, which may be considered as the space of lines through the origin in $\R^2$, or perhaps more concretely as a M\"obius band sewn to a disk along their $S^1$ boundaries. The Klein bottle is the connected sum of two copies of $\RP^2$.   More generally one can take the connected sum of $g$ copies of $\RP^2$ to form the genus $g$ nonorientable surface.

This is a complete list of closed surfaces. Their Euler characteristics $\chi$ are given by the formulas:

Orientable surfaces of genus $g$ (where $g = 0$ for $S^2$): $\chi = 2 - 2g$.

Nonorientable surfaces of genus $g$: $\chi = 2 - g$.

\begin{proposition}
Two closed surfaces are homeomorphic if and only if they are both orientable or both nonorientable, and they have the same genus.
\end{proposition}

We note the well-known fact that the connected sum of $T^2$ with $\RP^2$ is homeomorphic with the connected sum of three copies of $\RP^2$, both being nonorientable of genus 3.  This surface is known as Dyck's surface.

The fundamental groups of closed surfaces have presentations:

If $\Sigma$ is the orientable surface of genus $g$:
$$\pi_1(\Sigma) \cong \langle a_1, b_1, \dots , a_g, b_g | \; [a_1, b_1][a_2, b_2] \cdots [a_g,b_g] = 1  \rangle$$

If $\Sigma$ is the nonorientable surface of genus $g$:

$$\pi_1(\Sigma) \cong \langle a_1, \dots , a_g| \; a_1^2 a_2^2 \cdots a_g^2= 1  \rangle$$

\begin{theorem}[\cite{RW}]\label{surface}
If $\Sigma$ is a surface, then $\pi_1(\Sigma)$ is bi-orderable, with two exceptions: the Klein bottle, whose group is only left-orderable and the projective plane which is not left-orderable because it is a torsion group of order 2.
\end{theorem}

\begin{proof}  First consider the Klein bottle group, which has the alternative presentation
$\langle x, y | y^{-1}xy = x^{-1} \rangle$.  If one kills the infinite cyclic subgroup generated by $x$ the resulting quotient is also infinite cyclic.  Thus the Klein bottle group is in the middle of a short exact sequence flanked by orderable groups, and is therefore left-orderable.  It cannot be bi-ordered because the defining relation would then imply the contradiction that $x$ is greater than the identity if and only $x^{-1}$ is greater than the identity.

We now sketch the proof that closed surfaces other than the Klein bottle and $\RP^2$ may be bi-ordered.  The proof boils down to the single case of the nonorientable surface $\Sigma$ of genus 3, known as Dyck's surface, by the trick of considering covering spaces.  Recall that the projection map of a covering space induces injective homomorphisms of the corresponding fundamental groups.  Now nonorientable surfaces of genus $g \ge 3$ may be realized as the connected sum of the torus $T^2$ with $g - 2$ copies of $\RP^2$.  In other words, one may remove $g-2$ disjoint disks from $T^2$ and replace them by M\"obius bands to construct such a surface.  By considering a finite covering of $T^2$ by itself, and lifting a disk downstairs to multiple disks upstairs, we can construct finite sheeted covers of $\Sigma$ by all higher genus nonorientable surfaces.  Thus their fundamental groups can be considered as (finite index) subgroups of $\pi_1(\Sigma)$.  As for the orientable surfaces, we consider the oriented double cover of a nonorientable surface of genus $g \ge 3$.  Since Euler characteristics double, we see that the oriented surface upstairs in the cover will have genus $g-1$.
Therefore $\pi_1(\Sigma)$ also contains subgroups isomorphic to the fundamental groups of orientable surfaces of genus 2 or more.  This leaves the torus to consider, but its group is just 
$\Z^2$ which is obviously bi-orderable.

It remains to order $\pi_1(\Sigma)$, where $\Sigma$ is Dyck's surface, which we can also do by regarding covering spaces.  Take the universal cover of $T^2$ by $\R^2$, and choose a small disk in $T^2$, which lifts to infinitely many disks in $\R^2$, which we may take centered at the integral points $(m, n) \in \R^2$.  Now remove all the disks downstairs and upstairs and replace them by 
M\"obius bands.  This produces an infinite-sheeted covering $\tilde{\Sigma} \to \Sigma$.  One calculates that 
$\pi_1(\tilde{\Sigma} )$ is an infinitely-generated free group $F_{\infty}$ with generators represented by the central curves of the M\'obius bands that were sewn to the punctured $\R^2$, connected by tails to some fixed basepoint.  Thus we have an exact sequence
$$1 \to F_{\infty} \to  \pi_1(\Sigma) \to \Z^2 \to 1$$
in which $\pi_1(\Sigma)$ is sandwiched between bi-ordered groups.   The free group $F_{\infty}$ may be ordered, for example using a Magnus expansion, in such a way that the ordering is invariant under conjugation by elements of $\pi_1(\Sigma)$, that is, deck transformations of the covering. (see \cite{RW} for details).  We then appeal to proposition \ref{exact}.

\end{proof}

Note that this corrects a statement in the literature \cite{Lev43}, p. 201 ``... the fundamental group of a one-sided surface cannot be ordered.''
 
\section{Three-dimensional manifolds}

We will see that  many, but certainly not all, fundamental groups of 3-manifolds are left-orderable, or even locally indicable.  Orderability of these groups sheds light on the foliations, mappings and other structures related to 3-manifold theory.
A very useful criterion for left-orderability is the theorem of Burns and Hale \cite{BH72} that reduces the question to a local one.

\begin{theorem}[Burns-Hale]\label{BH}
A group $G$ is left-orderable if and only if for every nontrivial finitely generated subgroup $H$ of $G$ there exists a nontrivial left-orderable group $L$ and a surjective homomorphism $H \to L$.
\end{theorem}

We refer the reader to \cite{BH72} for a proof.  This result implies that locally indicable groups are left-orderable.  We'll see shortly that the converse does not hold.

The Burns-Hale theorem is the basic ingredient for a powerful criterion for orderability of fundamental groups of 3-dimensional manifolds.  A 3-dimensional manifold $M$ is said to be {\em irreducible} if every (smooth) 2-dimensional sphere in the manifold separates $M$ into two parts, one of which is homeomorphic with a 3-dimensional ball.  As with 2-manifolds (and knots) there is a unique prime decomposition theorem, due to Milnor \cite{Milnor62} which states that every compact 3-manifold is uniquely expressible as a connected sum of prime factors.  The connected sum of two 3-manifolds is obtained by deleting a small 3-ball from each and attaching the remaining parts to each other along the resulting 2-sphere boundaries.  The connected sum of a manifold with $S^3$ results in the same manifold (up to homeomorphism), and a manifold is called {\em prime} if it is not the connected sum of two manifolds, neither being $S^3$.  An irreducible manifold is prime.  However, there is exactly one prime closed manifold which is not irreducible, namely $S^2 \times S^1$, which contains a sphere ($S^2 \times *$) which does not separate the manifold.  Of course its fundamental group, being infinite cyclic, is bi-orderable.

\begin{theorem}[\cite{BRW}] \label{fund}  Suppose $M$ is an irreducible orientable 3-manifold, possibly noncompact and possibly with boundary.  Then $\pi_1(M)$ is left-orderable if and only if there exists a nontrivial left-orderable group $L$ and a surjective homomorphism $\phi: \pi_1(M) \to L$.
\end{theorem}

We sketch a proof (which appears in \cite{BRW}) using Theorem \ref{BH} and a technique due to Howie and Short \cite{HS}.  Consider a finitely generated subgroup $H$ of  $\pi_1(M)$.  If the index of $H$ in  $\pi_1(M)$ is finite, then the restriction of $\phi$ to $H$ provides a nontrivial homomorphism to the nontrivial left-orderable group $\phi(H)$ which is finite index in $L$.  The interesting case is when the index of $H$ in  $\pi_1(M)$ is infinite.  Then there exists an infinite sheeted covering space $p:\tilde{M} \to M$, such that $p_*\pi_1(\tilde{M}) = H$.  Recall that 
$p_*$ is injective, being induced by a covering map (basepoints are being suppressed for simplicity of notation).  Since $\pi_1(\tilde{M})$ is finitely generated, a theorem of P. Scott \cite{Sco} asserts there exists a {\em compact} submanifold $C$ of 
$\tilde{M}$ so that the inclusion induces an isomorphism $i_* : \pi_1(C) \to \pi_1(\tilde{M})$.  Being compact, and lying in the noncompact manifold $\tilde{M}$, $C$ necessarily has nonempty boundary.  We then use irreducibility to argue that we may assume that $\partial C$ contains no 2-sphere components, and thus $\partial C$ has infinite first homology groups.  A standard argument than shows that $C$ itself has infinite first homology groups.  The Hurewicz homomorphism $h:\pi_1(C) \to H_1(C)$ has image an infinite, finitely generated abelian group, so one can construct a surjective homomorphism $\pi_1(C) \to \Z$.  But 
$\pi_1(C) \cong \pi_1(\tilde{M}) \cong H$, so $H$ maps onto the left-orderable group $\Z$.  \qed

A very similar argument establishes the following.

\begin{theorem}  If the irreducible orientable 3-manifold $M$ has positive first Betti number, or in other words the abelianization of $\pi_1(M)$ is infinite, then  $\pi_1(M)$ is locally indicable.
\end{theorem}

Irreducibility is not a very restrictive assumption.  For 3-manifolds the fundamental group of a connected sum of 3-manifolds is just the free product of the groups of the summands.  If all terms of a connected sum are left-orderable, then so is the sum, and similarly for bi-orderability, by applying Proposition \ref{freeproduct}.

If $K$ is a knot in $S^3 \cong \R^3 \cup \infty$, one defines its knot group to be the fundamental group of its complement $\pi_1(S^3 \setminus K)$.  Knot complements are irreducible and by Alexander duality $H_1(S^3 \setminus K) \cong \Z$.  So we conclude the following.

\begin{corollary}
All knot groups are locally indicable, and hence left-orderable.
\end{corollary}

Some knot groups are bi-orderable -- for example the knot $4_1$ commonly known as the figure eight knot has bi-orderable group, as shown in \cite{PR}.  On the other hand, the trefoil $3_1$ does {\em not} have bi-orderable group.  As it happens the group of the trefoil is isomorphic with $B_3$, which was noted to be non-bi-orderable in Proposition \ref{braidsnotbo}.  

More generally, consider the torus knot $K_{p,q}$, where $p$ and $q$ are relatively prime integers.  This is by definition a curve on the surface of a standard donut in $\R^3$ which representing the homology class 
$p\lambda + q\mu$ where $\lambda$ and $\mu$ are the standard ``longitude'' and ``meridian'' curves of the surface. 

\begin{proposition}
If $T_{p,q}$ is a nontrivial torus knot, its group is not bi-orderable.
\end{proposition}

One way to see this is to look at its group $\langle a, b | a^p = b^q \rangle$.  Notice that $a$ commutes with $b^q$ so by Lemma \ref{power commutes}, if the group were bi-orderable we would conclude that the group is abelian.  A standard result of knot theory is that only the trivial knot has abelian knot group (namely $\Z$).

The papers \cite{CR} and \cite{PR} discuss general conditions for determining whether certain knot groups are bi-orderable, but the bi-orderablity question for most knots remains open.

\begin{example} Bergman \cite{Ber} described a 3-manifold whose fundamental group is left-orderable, but not locally indicable.  It may be described as the Brieskorn manifold 
$\Sigma(2, 3, 7)$, which is the intersection of the unit 5-sphere in complex 3-dimensional space 
$\C^3$ with complex coordinates $u, v, w$, and the variety defined by $u^2 + v^3 + w^7 = 0.$
It can also be described as a certain surgery on the trefoil knot (see \cite{PS97}, pp. 116-117).  Its fundamental group has the presentation
$$\pi_1(\Sigma(2, 3, 7)) \cong \langle a, b | a^7 = b^3 = (ba)^2 \rangle.$$

One can easily check that this group is perfect, that is it abelianizes to the trivial group.  In topologists jargon, it is a ``homology sphere,'' as it has the same integral homology groups as $S^3$.  It follows that $\pi_1(\Sigma(2, 3, 7))$ is not indicable and so certainly not locally indicable.  Nevertheless, Bergman gave an explicit embedding of this group in the group 
$\widetilde{\PSL(2, \R)}$, which is the universal cover of the group $\PSL(2, \R)$.  Now $\PSL(2, \R)$ acts on the circle $S^1$, for example by fractional linear transformations on $\R \cup \infty \cong S^1$.  Moreover, as a space $\PSL(2, \R)$ has the homotopy type of $S^1$, so its universal cover 
$\widetilde{\PSL(2, \R)}$ is an infinite cyclic cover, and is a group which acts effectively on the real line.  From Proposition \ref{action} we can conclude that 
$\widetilde{\PSL(2, \R)}$ is a left-orderable group, and ditto for  $\pi_1(\Sigma(2, 3, 7))$.
\end{example}

The above example is a manifold whose geometry is modeled on $\widetilde{\PSL(2, \R)}$, one of the eight Thurston geometries in dimension three.  The most important of the geometries is hyperbolic.  A closed manifold is said to be {\em hyperbolic} if it has the universal cover hyperbolic 3-space $\H^3$ so that the deck transformations are hyperbolic isometries.  One can use this to define the {\em volume} of the manifold, and by Mostow rigidity the volume is actually a topological invariant.  According to Thurston, the set of volumes of hyperbolic 3-manifolds is a well-ordered set of real numbers.

\begin{example}  The Weeks manifold $W$ is the closed hyperbolic 3-manifold of least volume
\cite{GMM}.  Its fundamental group is
$$\pi_1(W) \cong \langle a, b | babab=ab^{-2}a, ababa=ba^{-2}b \rangle.$$

The abelianization of this group is isomorphic with $\Z / 5\Z \oplus \Z / 5\Z$.  So the group is not perfect, but $W$ is a rational homology sphere, in that its homology with rational coefficients agrees with that of $S^3$.  It is shown explicitly in \cite{CD03} that this group cannot be left-ordered.  The defining relations can be rewritten as $b^{-1}ab^{-2}a=(ab)^2=ba^{-2}ba^{-1}$ and $a^{-1}ba^{-2}b=(ba)^2=ab^{-2}ab^{-1}$ so one may argue as follows: we may assume that $a>1$ if such a left-order were to exist and get a contradiction to each of the cases $b<1$, $a>b>1$ and $b>a>1$ (as the reader may easily check).
\end{example}

 An important reason for studying knots is that $3$-manifolds may be constructed from knots by a process called surgery: one removes a neighbourhood of a knot in $S^3$ and replaces it with a solid torus attached in one of infinitely many possible ways.  In fact all closed orientable 3-manifolds arise by this construction, if one uses a finite disjoint collection of knots for the surgeries.  The following is a sample application of orderability to this situation.

\begin{theorem}[\cite{CR}]
If $K \subset S^3$ is a knot with bi-orderable knot group, then surgery on $K$ cannot produce a manifold with finite fundamental group.
\end{theorem}

In fact, it is shown in \cite{CR} more generally that surgery on such a knot cannot produce what Ozsv{\'a}th and Szab{\'o} \cite{OS} call an L - space, that is, a rational homology sphere whose Heegaard-Floer homology has the smallest possible rank.

It has recently been conjectured by Boyer, Gordon and Watson \cite{BGW11} that a rational homology 3-sphere is an L-space if and only if its fundamental group is {\em not} left-orderable.  They present considerable evidence for this conjecture, indeed it has been verified for manifolds modeled on  seven of the eight Thurston 3-manifold geometries, leaving the hyperbolic case open, in general, at the time of this writing.  This is a fascinating connection between the very powerful new 3-manifold invariant of Heegaard-Floer homology and the orderability of groups.

\subsection{Foliations} 

Let us now consider (codimension one) foliations of 3-manifolds $M$.  By this we mean a collection 
$\F$ of subsets of $M$ for which appropriate $\R^3$ charts at points of $M$ meet members of $\F$ in parallel planes in $\R^3$.  Members of $\F$  are called leaves: they may be closed surfaces, or they may be noncompact and wrap around and meet the chart  infinitely many times.    It is known that every closed 3-manifold admits such foliations.  $\F$ is said to be transversely oriented if there is a continuous assignment of normal vectors to all the leaves.  If each member of $\F$ is considered a point, with the natural decomposition space topology, one gets the ``space of leaves'' which may be a non-Hausdorff space.   

If $\tilde{M} \to M$ is a covering space, then a foliation $\F$ of $M$ naturally lifts to a foliation 
$\tilde{\F}$ of $\tilde{M}$.  An $\R$-covered foliation $\F$ of $M$ is one which, when lifted to the universal cover of $M$ becomes a foliation whose space of leaves is homeomorphic with the real line $\R$.  

\begin{theorem}  If the 3-manifold $M$ has a transversely-oriented $\R$-covered foliation, then 
$\pi_1(M)$ is left-orderable.
\end{theorem}

This can be seen by noting that $\pi_1(M)$ acts by deck transformations on the universal cover 
$\tilde{M}$ and therefore permutes the leaves of $\tilde{\F}$.  In other words, it acts on the space of leaves, assumed homeomorphic to $\R$, and by orientation-preserving homeomorphisms because of the transverse orientation which also lifts in an equivariant way.  Then apply Proposition \ref{action}.

It follows, for example, that Weeks' manifold does not support a transversely-oriented $\R$-covered foliation.   An important class of foliations are the so-called {\em taut} foliations which means that for each leaf there is a simple closed curve in the manifold intersecting that leaf and everywhere transverse to the foliation.  The first examples of hyperbolic manifolds without taut foliations were given by Roberts, Shareshian and Stein \cite{RSS03} showing that their groups cannot act on the space of leaves, which in their case may be a possibly non-Hausdorff one-dimensional manifold.   The interested reader can pursue the fascinating interplay of foliations and orderability (including circular orders) in \cite{CD03} and \cite{Cal07}.

\subsection{Mappings of nonzero degree}

If $M$ and $N$ are connected oriented manifolds of the same dimension $n$ and $f: M \to N$ is a mapping, the {\em degree} of $f$ is determined by the homology mapping 
$f_* : H_n(M; \Z) \to  H_n(N; \Z)$.  In particular, each of those top-dimensional homology groups is canonically isomorphic to $\Z$, coming from specified orientations.  If $c \in H_n(M; \Z)$ is the preferred generator, then $f_*(c) \in H_n(N; \Z) \cong \Z$ is the degree of $f.$   Degree is a measure of the algebraic number of preimages of a generic point.  A constant map, or more generally one with a contractible image, of course has degree zero.

It is often of interest to ask whether mappings of nonzero degree exist between given manifolds.  If the target is the sphere of appropriate dimension, such maps always exist.  Indeed, given any manifold $M$ of dimension $n$, consider a smooth closed $n$-ball $B \subset M$, say a closed neighbourhood of a point.  If we smash the boundary of $B$, as well as everything outside of $B$ in $M$, to a single point, the resulting space is topologically an $n$-sphere.   Moreover the quotient mapping $M \to S^n$ has degree one.  In fact, by composing by degree $d$ maps $S^n \to S^n$, there are maps of any given degree when the target is a sphere.  

A connected sum $M_1 \sharp M_2$ always maps with degree 1 on each of its factors, simply by smashing the other factor to a point, so results assuming irreducibility often generalize. 
However, in general, maps of nonzero degree might not exist.  Orderability gives one obstruction to their existence.

\begin{theorem}\label{degree}
Suppose $M$ is a closed oriented irreducible 3-manifold whose fundamental group is not left-orderable and that $N$ is a closed oriented 3-manifold whose group {\em is} left-orderable.
Then maps $M \to N$ of nonzero degree do not exist.
\end{theorem}

Consider $f: M \to N$.  Then our assumptions ensure that the induced map 
$\pi_1(M) \to \pi_1(N)$ must be trivial, because otherwise Theorem \ref{fund} would imply that $\pi_1(M)$ is left-orderable, contradicting the hypothesis. Since the induced map on fundamental groups is trivial, standard covering space theory implies $f$ lifts to the universal cover $\tilde{N}$ which is noncompact.  Then we have a factorization
$H_3(M) \to H_3(\tilde{N}) \to H_3(N)$ of the homology map induced by $f$ in which the middle group is trivial, because $\tilde{N}$ is noncompact.  It follows that $deg(f) = 0.$ \qed

\begin{example}
Theorem \ref{degree} implies, for example, that there is no nonzero degree mapping from the Weeks manifold to the homology sphere $\Sigma(2, 3, 7)$. 
\end{example}

\subsection{Conjectures of Waldhausen and Thurston}

Group orderability is connected with certain deep conjectures about 3-manifolds due to Waldhausen and W. Thurston.  A {\em Haken} 3-manifold $M$ is one which contains an {\em incompressible} surface $F$, meaning a surface of genus $\ge 1$ in $M$ for which the inclusion induces an {\em injective} homomorphism $\pi_1(F) \to \pi_1(M)$.  Many questions regarding 3-manifold groups had been proved for Haken manifolds, often by inductive arguments involving cutting $M$ open along $F$ producing a ``simpler'' Haken manifold.  Not all 3-manifolds are Haken, but Waldhausen famously asked whether 3-manifolds are {\em virtually} Haken, meaning some finite-sheeted covering is a Haken manifold -- a question which remained open for decades.

Even more audaciously, Thurston proposed a stronger conjecture for the most important, and difficult, class of 3-manifolds -- hyperbolic ones.  He conjectured that they are virtually fibred. 
A 3-manifold $M$ is said to be {\em fibred} if there is a locally trivial fibre bundle map $M \to S^1$ with fibre a compact orientable surface.    This is a very strong type of foliation in which the leaves are surfaces, all topologically equivalent, and the space of leaves is topologically a circle.  

There is an exact sequence associated with fibrations, which in the case of a fibred 3-manifold $M$ with fibre $F$ reduces to
$$1 \to \pi_1(F) \to \pi_1(M) \to \pi_1(S^1) \to 1 .$$
So we see that fibred 3-manifolds are Haken.
Of course $\pi_1(S^1)$ is infinite cyclic and we have seen that $\pi_1(F)$ is also bi-orderable (Theorem \ref{surface}).  From Proposition \ref{exact} it follows that $\pi_1(M)$ is left-orderable if $M$ is fibred.  

Therefore there was a (faint) hope of finding a counterexample to the virtual fibred conjecture by finding a Kleinian group which is not virtually left-orderable, meaning no finite index subgroup is left-orderable.  That hope was recently dashed by stunning work of Agol \cite{Agol}, building on results of Haglund, Wise and others, in which he proved both the virtual Haken conjecture and the virtual fibering conjecture.  Moreover, he showed that if $M$ is hyperbolic, then $\pi_1(M)$ contains a finite-index subgroup which is a {\em right-angled Artin group}, also known as RAAG.  A RAAG is defined as having a finite set of generators and only relations saying that some of the generators commute with each other -- a kind of blend of free group and free abelian group.  Since it is known that every RAAG is bi-orderable, we conclude the following.

\begin{theorem} Every Kleinian group is virtually bi-orderable.   That is, if $M$ is a closed hyperbolic 3-manifold, some finite index subgroup of $\pi_1(M)$ is bi-orderable.
\end{theorem}

\section{Spaces of homeomorphisms} 

In this final section I would like to touch on some known results as well as research currently under way by myself and Danny Calegari on spaces of homeomorphisms.
Suppose $X$ is a topological space with closed subset $Y.$  We denote by $\Homeo (X, Y)$ the group of homeomorphisms $X \to X$ which are pointwise fixed on $Y$, the group operation being composition of functions.  $\Homeo (X, Y)$ can also be endowed with a topology, which we will ignore here, but rather concentrate on algebraic and orderability properties.  If $X$ is a simplicial complex or piecewise-linear manifold and $Y$ a PL closed subset, we consider the subgroup 
$\PL(X, Y)$ of homeomorphisms which are linear on each simplex of some finite subdivision of 
$X$.

\begin{proposition}
$\Homeo (I, \partial I)$ is left-orderable.
\end{proposition}

This is because $\Homeo (I, \partial I)$ is clearly isomorphic with $\Homeo_+(\R)$ which we already have seen to be left-orderable and indeed universal for countable left-orderable groups.  
The following was observed by Chehata \cite{Che52}.

\begin{proposition}[Chehata]
$\PL(I, \partial I)$ is bi-orderable.
\end{proposition}

It should be emphasized that each element of $\PL(I, \partial I)$ is a function which has only finitely many breaks where it may change slope.  We can define the positive cone to be the collection of all PL homeomorphisms whose graph $\{(t, f(t))\}$ in $I \times I$ has first departure from the diagonal veering above (rather than below) the diagonal.  \qed

Let us consider the 2-dimensional analogue.  The following is classical.

\begin{theorem}[Kerekjarto, Brouwer, Eilenberg]
$\Homeo (I^2, \partial I^2)$ is torsion-free. 
\end{theorem}

I believe it is an open question whether $\Homeo (I^2, \partial I^2)$ is left-orderable.

\begin{theorem}[Calegari-Rolfsen]\label{pl2}
$\PL(I^2, \partial I^2)$ is locally-indicable, and therefore left-orderable.
\end{theorem}

Here is an outline of the proof.  Consider a nontrivial subgroup $H$ of $\PL(I^2, \partial I^2)$ generated by the finite set $h_1, \dots, h_k$ of functions.  The fixed point set $fix(h_i)$ of each generator is a finite polyhedron in $I^2$ containing  $\partial I^2$, and the same may be said of the global fixed point set  $fix(H) = \cap_{i=1}^k fix(h_i)$.  Now we choose a point $p$ which is on the frontier of $fix(H)$; we can arrange that $p$ has a neighbourhood which can be identified with 
$\R^2 = \{(x, y)\}$ in such a way that all the functions are the identity on the $x$-axis and act linearly on the upper half plane.   We then map $H$ nontrivially to the ``germs'' of functions of $H$ at $p$ which according to the following lemma is a locally indicable group.  It follows that $H$ is indicable.
\qed

\begin{lemma}
The group of functions $\R^2_+ \to \R^2_+$ which are linear, and equal to the identity on the boundary, is isomorphic with the semidirect product of two locally indicable groups, and is therefore locally indicable.
\end{lemma}

Indeed such a function corresponds to a matrix
$\left( \begin{array}{cc}
1 & s \\
0 & r  \\
\end{array} \right)$ with $s \in \R$ and $r \in \R_+$.  Thus we have an isomorphism with the semidirect product of $\R$ as an additive group and $\R_+$ as a multiplicative group. \qed

\medskip
Theorem \ref{pl2} has been generalized to higher dimensions and more general manifolds in forthcoming work with Calegari.

\bibliographystyle{plain}

\bibliography{OrdGpsAndLDTopology}

\def\cprime{$'$} \def\cprime{$'$}
\begin{thebibliography}{10}

\bibitem{Agol}
Ian Agol.
\newblock The virtual haken conjecture.
\newblock 2012.
\newblock Preprint, arXiv:1204.2810v1.

\bibitem{Artin}
Emil Artin.
\newblock The theory of braids.
\newblock {\em American Scientist}, 38:112--119, 1950.

\bibitem{Ber}
George~M. Bergman.
\newblock Right orderable groups that are not locally indicable.
\newblock {\em Pacific J. Math.}, 147(2):243--248, 1991.

\bibitem{MR77}
Roberta Botto~Mura and Akbar Rhemtulla.
\newblock {\em Orderable groups}.
\newblock Marcel Dekker Inc., New York, 1977.
\newblock Lecture Notes in Pure and Applied Mathematics, Vol. 27.

\bibitem{BGW11}
Steven Boyer, Cameron~McA Gordon, and Liam Watson.
\newblock On l-spaces and left-orderable fundamental groups.
\newblock 2011.
\newblock Preprint, arXiv:1107.5016.

\bibitem{BRW}
Steven Boyer, Dale Rolfsen, and Bert Wiest.
\newblock Orderable 3-manifold groups.
\newblock {\em Ann. Inst. Fourier (Grenoble)}, 55(1):243--288, 2005.

\bibitem{Bro84}
S.~D. Brodski{\u\i}.
\newblock Equations over groups, and groups with one defining relation.
\newblock {\em Sibirsk. Mat. Zh.}, 25(2):84--103, 1984.

\bibitem{BH72}
R.~G. Burns and V.~W.~D. Hale.
\newblock A note on group rings of certain torsion-free groups.
\newblock {\em Canad. Math. Bull.}, 15:441--445, 1972.

\bibitem{Buttsworth71}
R.~N. Buttsworth.
\newblock A family of groups with a countable infinity of full orders.
\newblock {\em Bull. Austral. Math. Soc.}, 4:97--104, 1971.

\bibitem{Cal07}
Danny Calegari.
\newblock {\em Foliations and the geometry of 3-manifolds}.
\newblock Oxford Mathematical Monographs. Oxford University Press, Oxford, UK,
  2007.

\bibitem{CD03}
Danny Calegari and Nathan~M. Dunfield.
\newblock Laminations and groups of homeomorphisms of the circle.
\newblock {\em Invent. Math.}, 152(1):149--204, 2003.

\bibitem{Che52}
C.~G. Chehata.
\newblock An algebraically simple ordered group.
\newblock {\em Proc. London Math. Soc. (3)}, 2:183--197, 1952.

\bibitem{CR}
Adam Clay and Dale Rolfsen.
\newblock Ordered groups, eigenvalues, knots, surgery and {$L$}-spaces.
\newblock {\em Math. Proc. Cambridge Philos. Soc.}, 152(1):115--129, 2012.

\bibitem{Conrad59}
Paul Conrad.
\newblock Right-ordered groups.
\newblock {\em Michigan Math. J.}, 6:267--275, 1959.

\bibitem{Deh}
Patrick Dehornoy.
\newblock Braid groups and left distributive operations.
\newblock {\em Trans. Amer. Math. Soc.}, 345(1):115--150, 1994.

\bibitem{DDRW}
Patrick Dehornoy, Ivan Dynnikov, Dale Rolfsen, and Bert Wiest.
\newblock {\em Ordering braids}, volume 148 of {\em Mathematical Surveys and
  Monographs}.
\newblock American Mathematical Society, Providence, RI, 2008.

\bibitem{DD}
T.~V. Dubrovina and N.~I. Dubrovin.
\newblock On braid groups.
\newblock {\em Mat. Sb.}, 192(5):53--64, 2001.

\bibitem{FR}
Michael Falk and Richard Randell.
\newblock Pure braid groups and products of free groups.
\newblock In {\em Braids ({S}anta {C}ruz, {CA}, 1986)}, volume~78 of {\em
  Contemp. Math.}, pages 217--228. Amer. Math. Soc., Providence, RI, 1988.

\bibitem{FGRRW99}
R.~Fenn, M.~T. Greene, D.~Rolfsen, C.~Rourke, and B.~Wiest.
\newblock Ordering the braid groups.
\newblock {\em Pacific J. Math.}, 191(1):49--74, 1999.

\bibitem{Fuc}
L.~Fuchs.
\newblock {\em Partially ordered algebraic systems}.
\newblock Pergamon Press, Oxford, 1963.

\bibitem{GMM}
David Gabai, Robert Meyerhoff, and Peter Milley.
\newblock Mom technology and volumes of hyperbolic 3-manifolds.
\newblock {\em Comment. Math. Helv.}, 86(1):145--188, 2011.

\bibitem{Glass}
A.~M.~W. Glass.
\newblock {\em Partially ordered groups}, volume~7 of {\em Series in Algebra}.
\newblock World Scientific Publishing Co. Inc., River Edge, NJ, 1999.

\bibitem{Holder01}
O.~H\"{o}lder.
\newblock {D}ie {A}xiome der {Q}uantität und die {L}ehre vom {M}ass.
\newblock {\em Ber. Verh. Sachs. Ges. Wiss. Leipzig Math. Phys. Cl.}, 53:1--64,
  1901.

\bibitem{HS}
James Howie and Hamish Short.
\newblock The band-sum problem.
\newblock {\em J. London Math. Soc. (2)}, 31(3):571--576, 1985.

\bibitem{Ito}
Tetsuya Ito.
\newblock Braid ordering and knot genus.
\newblock {\em J. Knot Theory Ramifications}, 20(9):1311--1323, 2011.

\bibitem{Jones}
V.~F.~R. Jones.
\newblock Hecke algebra representations of braid groups and link polynomials.
\newblock {\em Ann. of Math. (2)}, 126(2):335--388, 1987.

\bibitem{KR}
Djun~Maximilian Kim and Dale Rolfsen.
\newblock An ordering for groups of pure braids and fibre-type hyperplane
  arrangements.
\newblock {\em Canad. J. Math.}, 55(4):822--838, 2003.

\bibitem{KM96}
Valeri{\u\i}~M. Kopytov and Nikola{\u\i}~Ya. Medvedev.
\newblock {\em Right-ordered groups}.
\newblock Siberian School of Algebra and Logic. Consultants Bureau, New York,
  1996.

\bibitem{Lev43}
F.~W. Levi.
\newblock Contributions to the theory of ordered groups.
\newblock {\em Proc. Indian Acad. Sci., Sect. A.}, 17:199--201, 1943.

\bibitem{Lin}
Peter~A. Linnell.
\newblock The space of left orders of a group is either finite or uncountable.
\newblock {\em Bull. Lond. Math. Soc.}, 43(1):200--202, 2011.

\bibitem{MN}
A.~V. Malyutin and N.~Yu. Netsvetaev.
\newblock Dehornoy order in the braid group and transformations of closed
  braids.
\newblock {\em Algebra i Analiz}, 15(3):170--187, 2003.

\bibitem{Milnor62}
J.~Milnor.
\newblock A unique decomposition theorem for {$3$}-manifolds.
\newblock {\em Amer. J. Math.}, 84:1--7, 1962.

\bibitem{Nav10b}
Andr{\'e}s Navas.
\newblock A finitely generated, locally indicable group with no faithful action
  by {$C^1$} diffeomorphisms of the interval.
\newblock {\em Geom. Topol.}, 14(1):573--584, 2010.

\bibitem{Nav}
Andr\'{e}s Navas.
\newblock On the dynamics of (left) orderable groups.
\newblock {\em Annales de l'institut Fourier}, 60(5):1685--1740, 2010.

\bibitem{NW11}
Andr{\'e}s Navas and Bert Wiest.
\newblock Nielsen-{T}hurston orders and the space of braid orderings.
\newblock {\em Bull. Lond. Math. Soc.}, 43(5):901--911, 2011.

\bibitem{OS}
Peter Ozsv{\'a}th and Zolt{\'a}n Szab{\'o}.
\newblock On knot {F}loer homology and lens space surgeries.
\newblock {\em Topology}, 44(6):1281--1300, 2005.

\bibitem{PR}
Bernard Perron and Dale Rolfsen.
\newblock On orderability of fibred knot groups.
\newblock {\em Math. Proc. Cambridge Philos. Soc.}, 135(1):147--153, 2003.

\bibitem{PS97}
V.~V. Prasolov and A.~B. Sossinsky.
\newblock {\em Knots, links, braids and 3-manifolds}, volume 154 of {\em
  Translations of Mathematical Monographs}.
\newblock American Mathematical Society, Providence, RI, 1997.
\newblock An introduction to the new invariants in low-dimensional topology,
  Translated from the Russian manuscript by Sossinsky [Sosinski{\u\i}].

\bibitem{RhemR}
Akbar Rhemtulla and Dale Rolfsen.
\newblock Local indicability in ordered groups: braids and elementary amenable
  groups.
\newblock {\em Proc. Amer. Math. Soc.}, 130(9):2569--2577 (electronic), 2002.

\bibitem{RSS03}
R.~Roberts, J.~Shareshian, and M.~Stein.
\newblock Infinitely many hyperbolic 3-manifolds which contain no {R}eebless
  foliation.
\newblock {\em J. Amer. Math. Soc.}, 16(3):639--679 (electronic), 2003.

\bibitem{KL}
Dale Rolfsen.
\newblock {\em Knots and links}, volume 346 of {\em AMS Chelsea Series}.
\newblock American Mathematical Society, Providence, RI, 2003.

\bibitem{RW}
Dale Rolfsen and Bert Wiest.
\newblock Free group automorphisms, invariant orderings and topological
  applications.
\newblock {\em Algebr. Geom. Topol.}, 1:311--320 (electronic), 2001.

\bibitem{RZ}
Dale Rolfsen and Jun Zhu.
\newblock Braids, orderings and zero divisors.
\newblock {\em J. Knot Theory Ramifications}, 7(6):837--841, 1998.

\bibitem{Sco}
G.~P. Scott.
\newblock Compact submanifolds of {$3$}-manifolds.
\newblock {\em J. London Math. Soc. (2)}, 7:246--250, 1973.

\bibitem{SW00}
Hamish Short and Bert Wiest.
\newblock Orderings of mapping class groups after {T}hurston.
\newblock {\em Enseign. Math. (2)}, 46(3-4):279--312, 2000.

\bibitem{Sikora04}
{\relax A.S}.~Sikora.
\newblock Topology on the spaces of orderings of groups.
\newblock {\em Bull. London Math. Soc.}, 36:519--526, 2004.

\bibitem{Thur}
William~P. Thurston.
\newblock {\em Three-dimensional geometry and topology. {V}ol. 1}, volume~35 of
  {\em Princeton Mathematical Series}.
\newblock Princeton University Press, Princeton, NJ, 1997.
\newblock Edited by Silvio Levy.

\bibitem{Vin}
A.~A. Vinogradov.
\newblock On the free product of ordered groups.
\newblock {\em Mat. Sbornik N.S.}, 25(67):163--168, 1949.

\end{thebibliography}

\end{document}